\documentclass{amsart}
\usepackage{amssymb}
\usepackage{graphicx}

\newtheorem{theorem}{Theorem}[section]
\newtheorem{lemma}[theorem]{Lemma}

\theoremstyle{definition}
\newtheorem{definition}[theorem]{Definition}
\newtheorem{example}[theorem]{Example}
\newtheorem{proposition}[theorem]{Proposition}

\theoremstyle{remark}

\numberwithin{equation}{section}
\let\angle\measuredangle

\begin{document}

\title{On the Curvature of Metric Triples}
\author{Qinglan Xia}
\address{University of California at Davis\\
Department of Mathematics\\
Davis,CA,95616}
\email{qlxia@math.ucdavis.edu}
\urladdr{http://math.ucdavis.edu/\symbol{126}qlxia}
\subjclass[2010]{Primary: 51F99, 53C23. Secondary: 46B85}
\keywords{Curvature of metric triple, CAT(k) space, Menger curvature, Steiner problem}

\begin{abstract}
In this article we introduce a notion of curvature, denoted by $ k_X(T)$, for a metric triple $T$ inside a (possibly discrete) metric space $X$.  
To define the notion, we employ the information consisting of side lengths of the triple as well as the minimum total distance from vertices of the triple to points of the metric space.  Those information provides us a unique number $k_X(T)$ such that the triple $T$ can be isometrically embedded into the model space $M_k^2$ up to $k\le k_X(T)$.  The value $k_X(T)$ agrees with the usual curvature when $X$ is a convex subset of a model space.  We also show that the curvature $k_X(T)$ of any metric triple $T$ inside a $CAT(k)$ space is bounded above by $k$.
\end{abstract}

\maketitle

\section{Introduction}


The purpose of this article is to introduce a notion of curvature for a metric triple inside a (possibly discrete) metric space.  Our motivation comes from considering the following problem: suppose $X$ is a (possibly finite) subset of an unknown metric space $Y$, how to discover geometric information of $Y$ from those of $X$? People are interested in finding some intrinsic and numerically computable geometric quantities of $X$ that could be used to indicate properties of the unknown ambient space $Y$.  As the rapid development of scientific computation, this problem becomes even more interesting since $X$ may simply be some scientific data collected from experiments or observations.  As an example, we investigate curvature information of $X$ in this article.


As one of the fundamental concepts in geometry, curvature has been studied extensively from those of smooth curves to Riemannian manifolds \cite{Lee}, to geodesic metric spaces (see \cite{BH},\cite{metricbook} and references therein), and beyond (e.g. \cite{forman}, \cite{sullivan}).    
In the literature, most works about curvature assume that the space itself is at least locally path-connected.
In the more general setting when the metric space is not necessarily path-connected, or even simply a finite set, one may consider Menger curvature of metric triples (see \cite{saucan} and references therein). Menger defined the curvature of a triple of points in a metric space as the reciprocal of the radius of the circle in the Euclidean plane which is circumscribed to a comparison configuration associated to that triple. Later Wald considered the curvature of quadruple of points in a metric space as the curvature $k$ of the model surface $M_{k}^{2}$ (i.e., the surface of constant curvature $k$) into which the metric quadruple can be isometrically embedded. 


In the definition of Menger curvature of a triple $T$, one uses a comparable triangle of the triple in the Euclidean space. Nevertheless, since the triple could also be isometrically embedded into other model spaces $M_{k}^{2}$ of constant curvature $k$, one may also consider comparison configurations of the triple in other model surfaces $M_{k}^{2}$ and get analogous concepts of curvature. In this sense, Menger's curvature is thoroughly Euclidean. 

On the other hand, existence of other points in $X$ may prohibit the triple $T$, together with other points, to be isometrically embedded into some $M_{k}^{2}$. This motives us to consider the minimum distance from the triple to points in $X$, and enable us to find a unique number $k_X\left( T\right) $ so that $T$ can only be isometrically embedded into $M_{k}^{2}$ up to $k\leq k_X\left( T\right) $.  

The triple-wise defined curvature $k_X(T)$ has some nice properties. 

\begin{itemize}
\item When $X$ is a convex subset of a model space $M_k^2$,  the value $k_X(T)$  agrees with the pointwise defined curvature $k$ of $X$, for each non-degenerate triple $T$ in $X$. 
\item The value $k_X(T)$ is intrinsically defined. It depends only on the side lengths of the triple as well as the minimum distance of the triple to $X$.
\item Suppose $X$ is a subset of an unknown metric space $Y$, then according to Proposition \ref{prop:subset}, $k_X(T) \ge k_Y(T)$. Adding more data of $Y$ to $X$ will decrease the gap between $k_X(T)$ and $k_Y(T)$, and provide a better approximation.
\item Suppose $X$ is a CAT(k) space. Then, according to Theorem \ref{main_thm}, every metric triple (with bounded perimeter) in $X$ will have curvature $k_X(T)$ bounded above by $k$. This result indicates that one could study properties of  a more general CAT(k)-type space $X$ by assuming that every triple-wise curvature $k_X(T)$ in $X$ is bounded above by $k$. Note that in this general setting, the space $X$ is not necessarily locally path connected.
\end{itemize}

%
%
%
%
%
%

The article is organized as follows. In section $\S 2$, we define the function $S(a,b,c,k)$ by studying the Steiner problem on the model space $M_k^2$. For a metric triple $T$  inside $M_k^2$ of side lengths $a,b,c$, the value $S(a,b,c,k)$ gives the minimum distance from the vertices of the triple to points in $M_k^2$.  In section $\S 3$, we calculate
the values of $S(a,b, c, k)$ numerically. In section $\S 4$, we define the curvature $k_X(T)$ of a metric triple  $T$ in any metric space $X$ using the function $S(a,b,c,k)$. We also investigate properties of $k_X(T)$ afterwards. In section $\S 5$, we show that in a CAT(k) space $X$, any metric triple $T$ (with bounded perimeter)  in $X$  will also have curvature $k_X(T)$ bounded above by $k$.

\section{The Steiner problem on triples in the model surfaces}

For a real number $k$, the model space $M_{k}^{2}$ with distance $\left\vert
\cdot \right\vert _{k}$ is the simply connected surface with constant
curvature $k$. That is, if $k=0$, then $M_{k}^{2}$ is the Euclidean plane.
If $k>0$, then $M_{k}^{2}$ is obtained from the sphere $\mathbb{S}^{2}$ by
multiplying the distance function by the constant $\frac{1}{\sqrt{k}}$. If $%
k<0$, then $M_{k}^{2}$ is obtained from the hyperbolic space $\mathbb{H}^{2}$
by multiplying the distance function by the constant $\frac{1}{\sqrt{-k}}$.
The diameter of $M_{k}^{2}$ is denoted by $D_{k}:=\pi /\sqrt{k}$ for $k>0$
and $D_{k}:=\infty $ for $k\leq 0$.

Suppose $a,b,c\in \left( 0,\infty \right) $ with $a\leq b\leq c\leq a+b$.
For any $k\in (-\infty ,\left( \frac{2\pi }{a+b+c}\right) ^{2}]$, i.e. $%
a+b+c\leq 2D_{k}$, there exists a triangle $\Delta ABC$ in $%
M_{k}^{2}$ with side length $\left\vert BC\right\vert
_{k}=a,\left\vert AC\right\vert _{k}=b$ and $\left\vert
AB\right\vert _{k}=c$. Then, we consider the Steiner problem of
minimizing
\begin{equation}
S\left( a,b,c,k\right) :=\min_{P\in \Delta ABC}\left\{\left\vert PA\right\vert _{k}+\left\vert PB\right\vert_{k}+\left\vert PC\right\vert_{k}\right\}
\label{comparison_function}
\end{equation}%
in the model space $M_{k}^{2}$. The minimum value is denoted by $S\left(
a,b,c,k\right) $, which is independent of the choice of the triangle $\Delta ABC$ in $%
M_{k}^{2}$.

In general, when the side lengths $a,b,c,$ is not necessarily increasingly ordered, one can simply extend the definition of $S$ by requiring $S(a,b,c,k)$ to be a symmetric function of variables $a,b,c$. Nevertheless, in the following context, we will simply assume that $a\le b\le c$.

In the space $M_k^2$, if the angle $\angle ACB \ge \frac{2\pi}{3}$, then the minimum value of $S$ is achieved at the vertex $C$. 
When the angle $\angle ACB < \frac{2\pi}{3}$, the minimum value of $S$ is achieved at an interior point $O$ of the triangle $\Delta ABC$. In this case, the point $O$ is called the Fermat's point of the triangle.
A useful fact about the Fermat's point $O$ is:  the angles $\angle A O B =\angle B O C=\angle C O A=\frac{2\pi}{3}$ in $M_k^2$.

\begin{lemma}
For any $a,b,c\in \left( 0,\infty \right) $ with $a\leq b\leq c\leq a+b$,
and any $k\in (-\infty ,\left( \frac{2\pi }{a+b+c}\right) ^{2}]$, it holds
that 
\begin{equation*}
\frac{a+b+c}{2}\leq S\left( a,b,c,k\right) \leq a+b.
\end{equation*}%
In particular, if $a+b=c$, then for any $k\in (-\infty ,\left( \frac{2\pi }{%
a+b+c}\right) ^{2}]$, 
\begin{equation*}
S\left( a,b,c,k\right) =a+b=c\text{.}
\end{equation*}
\end{lemma}

\begin{proof}
Clearly, by the triangle inequality, for each $P\in \Delta ABC$,
\begin{equation*}
2\left\{\left\vert PA\right\vert _{k}+\left\vert PB\right\vert_{k}+\left\vert PC\right\vert_{k}\right\}\geq \left\vert
BC\right\vert _{k}+\left\vert AC\right\vert _{k}+\left\vert
AB\right\vert _{k}\text{.}
\end{equation*}%
Thus, $2S\left( a,b,c,k\right) \geq a+b+c$. Also,
\begin{equation*}
S\left( a,b,c,k\right) \leq \left\vert CA\right\vert _{k}+\left\vert CB\right\vert_{k}=a+b.
\end{equation*}
\end{proof}

The function $S$ has the following properties:

\begin{proposition}
Let $S$ be the function defined by (\ref{comparison_function}). Then, 
\begin{enumerate}
\item For any $\lambda >0$, we have
\begin{equation}
S\left( \lambda a,\lambda b,\lambda c, \frac{k}{\lambda^2}\right) =\lambda S\left(
a,b,c,k\right).
\label{eqn: lambda}
\end{equation}
In particular, 
\begin{equation*}
\left( a+b+c\right) S\left( \frac{a}{a+b+c},\frac{b}{a+b+c},\frac{c}{a+b+c}%
,\left( a+b+c\right) ^{2}k\right) =S\left( a,b,c,k\right).
\end{equation*}
\item If $k>0$, then
\begin{equation}
S\left( a,b,c,k\right) =\frac{1}{\sqrt{k}}S\left( a\sqrt{k},b\sqrt{k},c\sqrt{%
k},1\right), \label{S_positive}
\end{equation}%
and
\begin{equation*}
S\left( \frac{a}{\sqrt{k}},\frac{b}{\sqrt{k}},\frac{c}{\sqrt{k}},k\right) =%
\frac{1}{\sqrt{k}}S\left( a,b,c,1\right) .
\end{equation*}%

\item If $k<0,$ then
\begin{equation}
S\left( a,b,c,k\right) =\frac{1}{\sqrt{-k}}S\left( a\sqrt{-k},b\sqrt{-k},c%
\sqrt{-k},-1\right),  \label{S_negative}
\end{equation}
and
\begin{equation*}
S\left( \frac{a}{\sqrt{-k}},\frac{b}{\sqrt{-k}},\frac{c}{\sqrt{-k}},k\right) =%
\frac{1}{\sqrt{-k}}S\left( a,b,c,1\right) .
\end{equation*}%

\end{enumerate}
\end{proposition}

%

\begin{proof}
These results follow from direct calculations using the definition of $S$ as well as the relationships between distance functions of $M_{k}^{2}$ and those of $\mathbb{S}^{2}$
or $\mathbb{H}^{2}$.
\end{proof}

\begin{definition}
For any $a,b,c\in \left( 0,\infty \right) $ with $a\leq b\leq c\leq a+b$, we consider the number $k^*$ defined as follows:

\begin{itemize}
\item If $c^{2}=a^{2}+b^{2}+ab=a^{2}+b^{2}-2ab\cos \left( \frac{2\pi }{3}\right) $, then $k^*=0$;
\item If $c^{2}<a^{2}+b^{2}+ab$, then
let $k^{\ast }>0$ be the number such that
\begin{equation*}
\cos \left( c\sqrt{k^{\ast }}\right) =\cos \left( a\sqrt{k^{\ast }}\right)
\cos \left( b\sqrt{k^{\ast }}\right) -\frac{1}{2}\sin \left( a\sqrt{k^{\ast }%
}\right) \sin \left( b\sqrt{k^{\ast }}\right);
\end{equation*}%
\item If $c^{2}>a^{2}+b^{2}+ab$,
then let $k^{\ast }<0$ be the number such that
\begin{equation*}
\cosh \left( c\sqrt{-k^{\ast }}\right) =\cosh \left( a\sqrt{-k^{\ast }}%
\right) \cosh \left( b\sqrt{-k^{\ast }}\right) +\frac{1}{2}\sinh \left( a%
\sqrt{-k^{\ast }}\right) \sinh \left( b\sqrt{-k^{\ast }}\right) .
\end{equation*}%
\end{itemize}
We denote the number $k^*$ by $\Lambda \left( a,b,c\right) $.
\end{definition}

%
%
From the definition of $\Lambda(a,b,c)$, it clearly holds that
\[\Lambda(ta,tb,tc)=\frac{1}{t^2}\Lambda(a,b,c)\]
for any $t>0$. As an example, in Figure \ref{figure:lambda}, we plot the graph of $\Lambda(a,b,c)$ with $a=1, b=1.2$ and $c$ varies from $(0.2, 2.2)$.

\begin{figure}[h]
\includegraphics[width=0.6\textwidth]{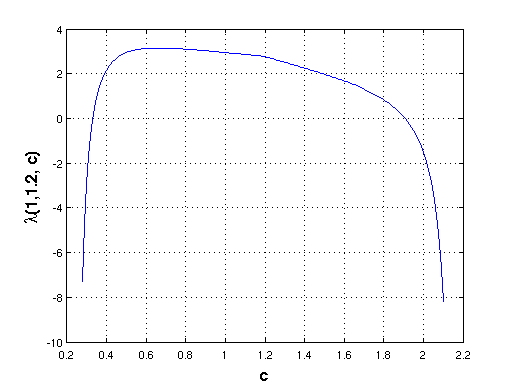}
\caption{Graph of the curve $\Lambda(a,b,c)$ with $a=1, b=1.2$ and $c$ varies from $(0.2, 2.2)$.  }
\label{figure:lambda}
\end{figure}



\section{Calculation of the function $S(a,b,c,k)$}
Let $a, b, c$ be fixed real values with $0<a\le b\le c \le a+b$. For any $k \in (\infty, \left( \frac{2\pi }{a+b+c}\right)^{2}]$, there exists a triangle $\Delta ABC$ in the model space $M_k^2$ with $a,b,c$ as its side lengths. We now consider the properties of $S(a,b,c,k)$ as a function of the variable $k \in (\infty, \left( \frac{2\pi }{a+b+c}\right)^{2}]$.

For any $k \in [\Lambda \left( a,b,c\right), \left( \frac{2\pi }{a+b+c}\right)^{2}] $, by the definition of $\Lambda(a,b,c)$, it follows that $\angle ACB \ge \frac{2\pi}{3}$ in $M_k^2$. In this case, the minimum value of $S$ in (\ref{comparison_function}) is achieved at the point $C$ and thus 
\begin{equation*}
S\left( a,b,c,k\right) =a+b.
\end{equation*}

For any $k\in (\infty, \Lambda \left( a,b,c\right))$, the angle $\angle ACB < \frac{2\pi}{3}$ in $M_k^2$, and the minimum value of $S$ in  (\ref{comparison_function}) is achieved at the Fermat point $O$ in the interior of the triangle $\Delta ABC$. In this case,
\begin{equation}
S(a,b,c,k)=x+y+z,
\label{eqn_xyz}
\end{equation}
where $x=\left\vert OA\right\vert _{k}, y=\left\vert OB\right\vert _{k}$ and $z=\left\vert OC\right\vert _{k}$. Note that the angles $\angle A O B =\angle B O C=\angle C O A=\frac{2\pi}{3}$ in $M_k^2$.

We now calculate the value of $S(a,b,c,k)$ in (\ref{eqn_xyz}) based on the signs of $k$.
\subsection{Calculation of $S(a,b,c,0)$} When $\Lambda \left( a,b,c\right)>0$, i.e.,  when $a^2+b^2+ab>c^2$, one can explicitly calculate the value of $S(a,b,c,0)$ as follows. 

In this case, $\Delta ABC$ is a triangle in the Euclidean plane. At the Fermat's point $O$, by the law of cosines, it holds that
\begin{eqnarray}
a^2 &=&y^2+z^2+yz, \\
b^2&=&x^2+z^2+xz,\\
c^2&=&x^2+y^2+xy. 
\end{eqnarray}
On the other hand, when calculate the area $\Delta$ of the triangle, we have
\[\frac{1}{2}xy\sin(2\pi/3)+\frac{1}{2}yz\sin(2\pi/3)+\frac{1}{2}xz\sin(2\pi/3)=\Delta \]
Thus, $xy+yz+zx=\frac{4\Delta}{\sqrt{3}}$.
As a result,
$2(x+y+z)^2=a^2+b^2+c^2+4\sqrt{3}\Delta$.
Therefore,
\[S(a,b,c,0)=x+y+z=\sqrt{\frac{a^2+b^2+c^2+4\sqrt{3}\Delta}{2}},\]
where the area $\Delta$ can be calculated by Heron's formula
\[\Delta=\sqrt{s(s-a)(s-b)(s-c)},\]
with $s=\frac{a+b+c}{2}$.

\subsection{Calculation of $S(a,b,c,k )$ with $k>0$}
To calculate $S\left( a,b,c,k\right) $ with $k>0$, by (\ref{S_positive}) and (\ref%
{S_negative}), it is sufficient to calculate $S\left( a,b,c,1\right) $.

In this case, $\Delta ABC$ is  a triangle in the unit sphere $\mathbb{S}^2$. When $\Lambda(a,b,c)>1$, i.e., when $\cos c > \cos a\cos b-\frac{1}{2}\sin a\sin b $, $S$ achieves its minimum at the Fermat's point $O$. 
%
%
%
Since the angles at $O$ are all $\frac{2\pi}{3}$, according to the spherical law of cosines, the values $x,y, z$ are giving by solving the following system of trigonometric equations:
\begin{eqnarray}
\cos a &=&\cos y\cos z-\frac{1}{2}\sin y\sin z,  \label{cos_a} \\
\cos b &=&\cos x\cos z-\frac{1}{2}\sin x\sin z , \label{cos_b} \\
\cos c &=&\cos x\cos y-\frac{1}{2}\sin x\sin y . \label{cos_c}
\end{eqnarray}%

\begin{proposition} \label{spherical}
For the $x, y, z,$ and $a, b, c$ given as above, let $X=\sin x, Y=\sin y, Z=\sin z$ and $u=\cos a, v=\cos b, w=\cos c$. Then, $X, Y, Z$
satisfy the following system of multivariate polynomial equations:
\begin{eqnarray}
&&\left( 3u^2+1\right) X^2+\left( 3v^2+1\right) Y^2 +\left( 6uv-2w\right)XY-3X^2Y^2+4w^2=D, \label{alg_c}\\
&&\left( 3u^2+1\right) X^2+\left( 3w^2+1\right) Z^2 +\left( 6uw-2v\right)XZ-3X^2Z^2+4v^2 =D, \label{alg_b}\\
&&\left( 3v^2+1\right) Y^2+\left( 3w^2+1\right) Z^2 +\left( 6vw-2u\right)YZ-3Y^2Z^2+4u^2=D, \label{alg_a}
\end{eqnarray}%
where $D=4u^2+4v^2+4w^2-8uvw $.
\end{proposition}
\begin{proof}
By symmetry, it is sufficient to prove the equation (\ref{alg_c}). Indeed, from (\ref{cos_a}) and (\ref{cos_b}), we have
\[
\cos a\cos x-\cos b\cos y=\frac{\sin z}{2}\sin \left( x-y\right)\]
and
\[
\cos a\sin x-\cos b\sin y=\sin \left( x-y\right) \cos z.\]

These two equations as well as (\ref{cos_c}) give 
\begin{eqnarray*}
& &\sin ^{2}\left( x-y\right)=\sin ^{2}\left( x-y\right)(\sin^2 z+\cos^ 2 z)\\
&=&4\left( \cos a\cos x-\cos b\cos y\right)
^{2}+\left( \cos a\sin x-\cos b\sin y\right) ^{2}\\
&=&4\cos ^{2}a\cos ^{2}x+4\cos ^{2}b\cos
^{2}y-8\cos a\cos b\cos x\cos y \\
&&+\cos ^{2}a\sin ^{2}x+\cos ^{2}b\sin ^{2}y-2\cos a\cos b\sin x\sin y \\
&=& 4\cos^2 a -3\cos^2a \sin^2 x +4\cos^2b -3\cos^2 b \sin^2y\\
& & -8\cos a \cos b (\cos c+\frac{1}{2}\sin x \sin y)-2\cos a\cos b \sin x \sin y\\
&=& 4\cos^2 a +4\cos^2b-3\cos^2a \sin^2 x  -3\cos^2 b \sin^2y\\
& &-8\cos a \cos b \cos c-6\cos a\cos b \sin x \sin y.
\end{eqnarray*}%

On the other hand, by using (\ref{cos_c}) again, we have
\begin{eqnarray*}
\sin ^{2}\left( x-y\right) &=&\left( \sin x\cos y-\cos x\sin y\right)
^{2}\\
&=&\sin ^{2}x\cos ^{2}y+\cos ^{2}x\sin ^{2}y-2\sin x\sin y\cos x\cos y \\
&=&\sin ^{2}x-2\sin ^{2}x\sin ^{2}y+\sin ^{2}y-2\sin x\sin y\left( \cos c+%
\frac{1}{2}\sin x\sin y\right) \\
&=&\sin ^{2}x-3\sin ^{2}x\sin ^{2}y+\sin ^{2}y-2\sin x\sin y\cos c.
\end{eqnarray*}%
Therefore, it follows that 
\begin{eqnarray*}
& &\sin ^{2}x-3\sin ^{2}x\sin ^{2}y+\sin ^{2}y-2\sin x\sin y\cos c\\
&=&4\cos^2 a +4\cos^2b-3\cos^2a \sin^2 x  -3\cos^2 b \sin^2y \\
& &-8\cos a \cos b \cos c-6\cos a\cos b \sin x \sin y.
\end{eqnarray*}
That is,
\begin{eqnarray*}
& &(1+3\cos^2a)\sin ^{2}x-3\sin ^{2}x\sin ^{2}y +(1+3\cos^2 b)\sin ^{2}y \\
& &+(6\cos a\cos b-2\cos c)\sin x\sin y\\
&=&4\cos^2 a +4\cos^2b-8\cos a \cos b \cos c,
\end{eqnarray*}
which is the equation (\ref{alg_c}).
\end{proof}

From either equations (\ref{cos_a}), (\ref{cos_b}), (\ref{cos_c}) or equations (\ref{alg_c}), (\ref{alg_b}), (\ref{alg_a}), one can calculate the numerical values of $S(a,b,c,1)$ via numerical methods
e.g. Newton's method. In Figure \ref{figure: S1}, we plot the graph of the parametric surface $S(a,b,c,1)$ with $a=1, b=(t-s)/2, c=(t+s)/2)$ for $t\in (a, 2\pi-a)$ and $s\in (-a, a)$.
\begin{figure}[t]
\includegraphics[width=0.8\textwidth]{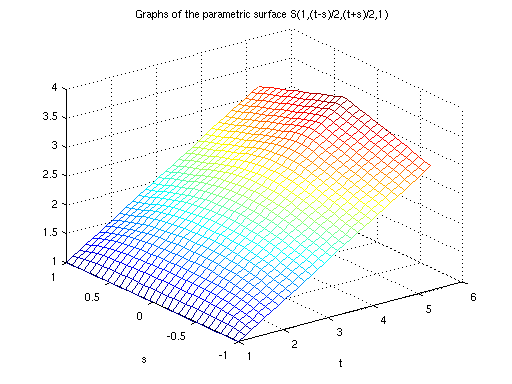}
\caption{Graph of the parametric surface $S(a,b,c,1)$ with $a=1, b=(t-s)/2, c=(t+s)/2)$ for $t\in (a, 2\pi-a)$ and $s\in (-a, a)$. }
\label{figure: S1}
\end{figure}

\subsection{Calculation of $S(a,b,c,-1)$}
Similarly, by the hyperbolic law of cosines, it follows that 
\begin{eqnarray}
\cosh a &=&\cosh y\cosh z+\frac{1}{2}\sinh y\sinh z  \label{cosh_a} \\
\cosh b &=&\cosh x\cosh z+\frac{1}{2}\sinh x\sinh z  \label{cosh_b} \\
\cosh c &=&\cosh x\cosh y+\frac{1}{2}\sinh x\sinh y  \label{cosh_c}
\end{eqnarray}%
for $x,y,z$ and then
$S\left( a,b,c,-1\right) =x+y+z.$

\begin{proposition}
For the $x, y, z,$ and $a, b, c$ given as above, let $X=\sinh x, Y=\sinh y, Z=\sinh z$ and $u=\cosh a, v=\cosh b, w=\cosh c$. Then, $X, Y, Z$
satisfy the following system of multivariate polynomial equations:
\begin{eqnarray}
&&\left( 3u^2+1\right) X^2+\left( 3v^2+1\right) Y^2 +\left( 6uv-2w\right)XY+3X^2Y^2-4w^2+D=0, \label{algh_c}\\
&&\left( 3u^2+1\right) X^2+\left( 3w^2+1\right) Z^2 +\left( 6uw-2v\right)XZ+3X^2Z^2-4v^2 +D=0, \label{algh_b}\\
&&\left( 3v^2+1\right) Y^2+\left( 3w^2+1\right) Z^2 +\left( 6vw-2u\right)YZ+3Y^2Z^2-4u^2+D=0, \label{algh_a}
\end{eqnarray}%
where $D=4u^2+4v^2+4w^2-8uvw $.
\end{proposition}
\begin{proof}
Follows from an analogous proof of Proposition \ref{spherical}.
\end{proof}
Again, from either equations (\ref{cosh_a}), (\ref{cosh_b}), (\ref{cosh_c}) or equations (\ref{algh_c}), (\ref{algh_b}), (\ref{algh_a}), one can calculate the numerical values of $S(a,b,c,-1)$ via Newton's method.

As a result, one can (numerically) calculate the values of $S(a,b,c,k)$ for all $k\in (-\infty, \Lambda(a,b,c)]$. In Figure \ref{figure:S}, we plot of the graph of $S(a,b,c,k)$  with $a=1, b=1.2$ and $c=1.3$. Note that $S(a,b,c,k)$ is a continuous strictly increasing function of $k\in (-\infty, \Lambda(a,b,c)]$.

\begin{figure}[h]
\includegraphics[width=0.6\textwidth]{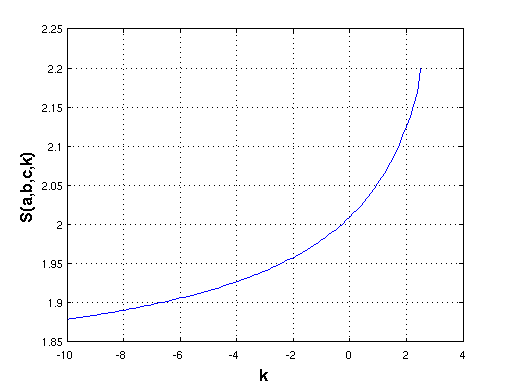}
\caption{Graph of the strictly increasing curve $S(a,b,c,k)$ with $a=1, b=1.2, c=1.3$ and the variable $k$ varies from $(-\infty, \Lambda(a,b,c)]$ with $\Lambda(a,b,c)=2.5081$ here. }
\label{figure:S}
\end{figure}
\section{Curvature of a metric triple}


\begin{definition}
Let $\left( X,d\right) $ be a metric space. A metric triple is a set $%
T=\left\{ p_{1},p_{2},p_{3}\right\} \subseteq X$, together with a set of
mutual distance $d_{ij}=d\left( p_{i},p_{j}\right) ;$ $1\leq i\leq j\leq 3$.
\end{definition}
Without losing of generality, we may assume that $d_{12}\le d_{23} \le d_{13}\le d_{12}+d_{23}$.
\begin{definition}
\bigskip For any metric triple $T$ in $(X,d)$, define
\begin{equation*}
g_X\left( T\right) :=\inf_{x\in X}\sum_{i=1}^{3}d\left(
p_{i},x\right) .
\end{equation*}
\end{definition}
Note that $g_X(T)\ge \frac{d_{12}+d_{23}+d_{13}}{2}$.
\begin{definition}
Let $T$ be a metric triple in a metric space $\left( X,d\right) $, and $S$ be the
function as defined in (\ref{comparison_function}). The curvature $%
k_{X}\left( T\right) $ of the metric triple $T$ with respect to $X$ is the
number $k\in \left[ -\infty ,\Lambda \left( d_{12},d_{23},d_{13}\right) %
\right] $ such that \[g_X\left( T\right) =S\left( d_{12},d_{23},d_{13},k\right).\]
\end{definition}

\begin{example}
Suppose $\left( X,d\right) $ is a convex subset in the model space $M_{k}^{2}$ for some real
number $k$, and $T$ is a metric triple in $X$. Then $g_X(T)=S(d_{12}, d_{23}, d_{31}, k)$. Thus, $k_X(T)\le k$. Moreover, if $k\le \Lambda(d_{12}, d_{23}, d_{31})$, i.e. the angle at each vertex of the corresponding triangle is no more than $2\pi/3$, then $k_X(T)=k$.
\end{example}
\begin{example}
 Let $T$ be a metric triple in a discrete metric space $(X,d)$ with $d\left( x,y\right)=1$ whenever $x\neq y$. Then, $g_X(T)=2$. In this case,
the curvature $k_X(T)=\Lambda(1,1,1)=3.6505=\left( \pi -\arccos \frac{1}{3}\right) ^{2}$. The comparable triangle in the model space $M_{k_X(T)}^2$ has an angle $2\pi/3$ at each of its vertices.
\end{example}
According to (\ref{eqn: lambda}), we immediately have the following proposition:
\begin{proposition}
Let $T=\{p_1,p_2,p_3\}$ be a metric triple in the metric space $(X, d_X)$, and $\tilde{T}=\{\tilde{p}_1, \tilde{p}_2, \tilde{p}_3\}$ be a metric triple in the metric space $(Y, d_Y)$. If for some $\lambda>0$, $d_Y(\tilde{p}_i, \tilde{p}_j)=\lambda d_X(p_i,p_j)$ for each $i,j =1,2,3$ and $g_Y(\tilde{T})=\lambda g_X(T)$, then $k_Y(\tilde{T})=\frac{1}{\lambda^2}k_X(T)$.
\end{proposition}

\begin{proposition}
\label{prop:subset}
Suppose $(X,d_X)$ can be isometrically embedded in $(Y,d_Y)$. Then, for any metric triple $T$ in $(X,d_X)$, we have $k_Y(T)\le k_X(T)$.
\end{proposition}
\begin{proof}
This is because $g_Y(T)\le g_X(T)$ and $S(d_{12}, d_{23}, d_{31}, k)$ is an strictly increasing function of $k$.
\end{proof}

\begin{lemma}
Let $T=\{p_1, p_2, p_3\}$ and $\tilde{T}=\{\tilde{p}_1, \tilde{p}_2, \tilde{p}_3\}$ be two metric triples in a metric space $(X,d)$. Then,
\begin{equation}
|g_X(T)-g_X(\tilde{T})|\le \sum_{i=1}^3 d(p_i, \tilde{p}_i). \label{eqn: g_X}
\end{equation}
\end{lemma}
\begin{proof}
For all $x\in X$, by the triangle inequality,
\[\sum_{i=1}^3 d(x, p_i) \le \sum_{i=1}^3 d(x, \tilde{p}_i)+\sum_{i=1}^3 d(p_i, \tilde{p}_i).\]
Taking the infimum over $x\in X$ on both sides of the inequality, we have $g_X(T)\le g_X(\tilde{T})+\sum_{i=1}^3 d(p_i, \tilde{p}_i)$. Similarly, $g_X(\tilde{T}) \le g_X(T)+\sum_{i=1}^3 d(p_i, \tilde{p}_i)$.
\end{proof}

\begin{proposition}
Let $T^{(n)}=\{p^{(n)}_1, p^{(n)}_2, p^{(n)}_3\}$ be a sequence of metric triples in a metric space $(X,d)$. If $T^{(n)}$ converges to a metric triple $T=\{p_1, p_2, p_3\}$ in $(X,d)$ in the sense that for each $i=1,2,3$,  $d(p^{(n)}_i, p_i)\rightarrow 0$ as $n\rightarrow \infty$, then 
\[\lim_{n\rightarrow \infty} k_X(T^{(n)})=k_X(T).\]
\end{proposition}
\begin{proof}
By (\ref{eqn: g_X}), $\lim_{n\rightarrow \infty}g_X(T^{(n)})=g_X(T)$. Since $S$ is also a continuous function, we have $\lim_{n\rightarrow \infty} k_X(T^{(n)})=k_X(T)$.
\end{proof}

\section{Curvature of metric triples in CAT(k) spaces}
For a real number $k$, a geodesic metric space $(X,d)$ is said to be CAT(k) if every geodesic triangle $\Delta ABC$ in $X$ with perimeter less than $2D_k$ satisfies the CAT(k) inequality. That is, let $\Delta A'B'C'$ be a comparison triangle in the model space $(M_k^2, \left\vert
\cdot \right\vert _k)$, with sides of the same length as the sides of $\Delta ABC$, such that for any $D$ in edge $BC$, there is a corresponding point $D'$ in the comparison edge $B'C'$ with $d(B,D)=|B'D'|_k, d(C,D)=|C'D'|_k$ and satisfies the inequality $d(A,D) \le |A'D'|_k$. 

Define
\[H_k=
\begin{cases}
\frac{\pi}{2\sqrt{k}}, & \mbox{if } k>0 \\
\infty, & \mbox{if }k=0 \\
\frac{1.3877}{\sqrt{-k}}, & \mbox{if } k<0.
\end{cases}
\]
\begin{theorem}
\label{main_thm}
Suppose $(X,d)$ is a CAT(k) space for some real number $k$. Then, for any metric triple $T=\{A,B,C\}$ in $X$ with $\max\{d(A,B), d(A,C), d(B,C)\} \le H_k$, it holds that
$k_X(T)\le k$.
\end{theorem}
\begin{proof}
It is sufficient to show that $g_X(T) \le g_{M^2_k}(T^*)$ where $T^*=\{A^*, B^*, C^*\}$ is the comparison configuration of $T$ in the model space $(M_k^2, \left\vert
\cdot \right\vert_k )$ of equal side lengths.  Without losing of generality, we may assume that there exists a Fermat's point $O^*$ of the triangle $A^*B^*C^*$ in $M^2_k$.  Then,
$g_{M^2_k}(T^*)=|A^*O^*|_k+|B^*O^*|_k+|C^*O^*|_k.$
Let $D^*$ be the point on the edge $B^*C^*$ such that $O*$ is located on the geodesic $A^*D^*$.  Suppose $|A^*O^*|_k=s|A^*D^*|_k$ and $|B^*D^*|_k=t|B^*C^*|_k$ for some $s,t \in(0,1)$. Then, let $D$ be the point on the edge $BC$ in the metric space $X$ corresponding to $D^*$, i.e. $d(B,D)=td(B,C)$. Also, let $O$ be the point on the geodesic $AD$ corresponding to $O^*$, i.e. $d(A,O)=sd(A,D)$. Note that 
$g_X(T)\le d(A,O)+d(B,O)+d(C,O)$.
By the CAT(k) inequality, it holds that $|A^*D^*|_k\ge d(A,D)=\sigma |A^*D^*|_k$ for some $\sigma \le 1$. Thus, it follows that $d(A,O)=\sigma |A^*O^*|_k$ and $d(O, D)=\sigma |O^*D^*|_k$.

Now, we consider a comparison triangle $A'B'D'$ (and $A'C'D'$) in $M_k^2$ of $ABD$ (and $ACD$, respectively) in $(X,d)$. Also, let $O'$ on $A'D'$ be the corresponding point of $O$ on $AD$ in $X$. So, $|A'O'|_k=d(A,O)=\sigma |A^*O^*|_k$. Again, by the CAT(k) inequality, we have $d(B,O)\le |B'O'|_k$ and $d(C,O)\le |C'O'|_k$. 
Thus,
\[g_X(T)\le d(A,O)+d(B,O)+d(C,O)\le |A'O'|_k+|B'O'|_k+|C'O'|_k.\]

Note that when $\sigma=1$, then one can simply take $A'=A^*, B'=B^*, C'=C^*, D'=D^*, O'=O^*$. In this case, $|A'O'|_k+|B'O'|_k+|C'O'|_k=|A^*O^*|_k+|B^*O^*|_k+|C^*O^*|_k=g_{M^2_k}(T^*)$.
Denote 
\begin{equation}
f(\sigma)=|A'O'|_k+|B'O'|_k+|C'O'|_k.
\label{function_f}
\end{equation}
for all $\sigma \in [0,1]$.  By Lemma \ref{lemma_zero_k}, Lemma \ref{lemma_positive_k} and Lemma \ref{lemma_negative_k} given below, we have $f(\sigma) \le f(1)$ whenever $\sigma  \in [0,1]$.  Thus, $g_X(T) \le g_{M^2_k}(T^*)$.
\end{proof}

\begin{lemma}
\label{lemma_zero_k}
Let $f$ be the function as given in (\ref{function_f}). When $k=0$, then $f(\sigma) \le f(1)$ whenever $\sigma  \in [0,1]$.
\end{lemma}

\begin{proof}
  In the case that $k=0$, we have $O'=(1-s)A'+sD'$ with $s=||A^*O^*||/||A^*D^*||$. Here, $||\cdot||=|\cdot|_0$ denotes the Euclidian distance on $M^2_0=\mathbb{R}^2$. Thus, by means of the law of cosines, we have
  \begin{eqnarray*}
  & &||B'O'||^2=||(1-s)B'A'+sB'D'||^2 \\
  &=&(1-s)^2||B'A'||^2+s^2||B'D'||^2+s(1-s)(||B'A'||^2+||B'D'||^2-||A'D'||^2)\\
   &=&(1-s)||B'A'||^2+s||B'D'||^2-s(1-s)||A'D'||^2\\
    &=&(1-s)||B^*A^*||^2+s||B^*D^*||^2-s(1-s)\sigma^2 ||A^*D^*||^2\\
  &=&\frac{1}{||A^*D^*||^2}(||D^*O^*||\times ||B^*A^*||^2+||A^*O^*||\times ||B^*D^*||^2)-||A^*O^*||\times ||A^*D^*||  \sigma^2.
  \end{eqnarray*}
  Similarly, we have
 \[ ||C'O'||^2=\frac{1}{||A^*D^*||^2}(||D^*O^*||\times ||C^*A^*||^2+||A^*O^*||\times ||C^*D^*||^2)-||A^*O^*||\times ||A^*D^*||  \sigma^2.\]
 As a result, 
\begin{eqnarray*}
 &&f(\sigma)=||A'O'||+||B'O'||+||C'O'|| \\
 &&=\sigma ||A^*O^*||+\sqrt{\frac{||A^*O^*|| \times ||B^*D^*||^2+||D^*O^*||\times ||A^*B^*||^2}{||A^*D^*||}-\sigma^2 ||A^*O^*|| \times ||O^*D^*||} \\
 &&+\sqrt{\frac{||A^*O^*||\times ||C^*D^*||^2+||D^*O^*|| \times||A^*C^*||^2}{||A^*D^*||}-\sigma^2 ||A^*O^*|| \times ||O^*D^*||}.
\end{eqnarray*}
Direct calculation yields that  $f''(\sigma)\le 0$  and 
\[f'(\sigma)=||A^*O^*||-\frac{||A^*O^*||\times ||O^*D^*||}{||B'O'||}\sigma -\frac{||A^*O^*||\times ||O^*D^*||}{||C'O'||}\sigma. \]
Thus,
\[f'(1)=||A^*O^*||-\frac{||A^*O^*||\times ||O^*D^*||}{||B^*O^*||} -\frac{||A^*O^*||\times ||O^*D^*||}{||C^*O^*||}.\]
Since $O^*$ is the Fermat point, the angles at $O^*$ are all $\frac{2 \pi}{3}$. Denote the angle $\angle O^*B^*D^*=\alpha$, then by the law of sine, we have 
\begin{eqnarray*}
f'(1)&=&||A^*O^*||(1-\frac{ ||O^*D^*||}{||B^*O^*||} -\frac{||O^*D^*||}{||C^*O^*||})\\
&=&||A^*O^*||(1-\frac{\sin \alpha}{\sin(\frac{2\pi}{3}-\alpha)} -\frac{\sin (\frac{\pi}{3}-\alpha)}{\sin(\frac{\pi}{3}+\alpha)}) \\
&=&||A^*O^*||(1-1)=0.
\end{eqnarray*}
Now, since $f''(\sigma)\le 0$ and $f'(1)=0$, it holds that $f'(\sigma) \ge  f'(1)=0$ for $\sigma \in [0,1]$. Thus, $f$ is increasing on $[0,1]$, and $f(\sigma) \le f(1)$ whenever $\sigma  \in [0,1]$.
\end{proof}

To consider the case when $k>0$, we need the following two lemmas first.  For simplicity, we let $|\cdot|$ denote the standard distance $|\cdot|_1$ on the unit sphere $\mathbb{S}^2$.
\begin{lemma}
Let $\Delta B'D'O'$ be a triangle in the unit sphere $(\mathbb{S}^2, |\cdot|)$. Then,
\begin{equation}
\cot(|O'B'|)=\cot(|O'D'|) \cos(\alpha)+\cot(\gamma)\frac{\sin(\alpha)}{\sin(|O'D'|)}
\label{cotangent}
\end{equation}
where $\alpha=\angle B'O'D' $ and $\gamma =\angle B'D'O'$.
\end{lemma}
\begin{proof}
By the spherical laws of cosines and sines, it follows that
\begin{eqnarray*}
& &\cot(|O'B'|)=\frac{\cos(|O'B'|)}{\sin(|O'B'|)}\\
&=&\frac{\cos(|B'D'|)\cos(|O'D'|)+\sin(|B'D'|)\sin(|O'D'|)\cos(\gamma)}{\sin(|O'B'|)}\\
&=&\frac{[\cos(|O'B'|)\cos(|O'D'|)+\sin(|O'B'|)\sin(|O'D'|)\cos(\alpha)]\cos(|O'D'|)}{\sin(|O'B'|)}\\
& &+\frac{\sin(|B'D'|)\sin(|O'D'|)\cos(\gamma)}{\sin(|O'B'|)}\\
&=&\cot(|O'B'|)\cos^2(|O'D'|)+\sin(|O'D'|)\cos(|O'D'|)\cos(\alpha)+\frac{\sin(\alpha)\sin(|O'D'|)}{\sin(\gamma)}\cos(\gamma).
\end{eqnarray*}
Simplifying it leads to (\ref{cotangent}). 
\end{proof}
\begin{lemma} 
\label{spherical_lemma}
Let $\Delta A^*B^*C^*$  be a triangle on the unit sphere $\mathbb{S}^2$ with a Fermat's point $O^*$ inside the triangle. Let $D^*$ be the point on the arc $B^*C^*$ such that $O^*$ is on the arc $A^*D^*$. 
Also, let $\Delta A'B'D'$ be a triangle on $\mathbb{S}^2$ such that $|A'B'|=|A^*B^*|, |B'D'|=|B^*D^*|$ and $|A'D'|=\sigma |A^*D^*|$ for some $\sigma \in (0,1]$. Let $O'$ be the point on the arc $A'D'$ such that $|A'O'|=\sigma|A^*O^*|$ and $|O'D'|=\sigma |O^*D^*|$. If $\max\{|A^*B^*|, |A^*C^*|, |B^*C^*|\}\le \frac{\pi}{2},$ then for all $\sigma \in[0,1]$, 
\begin{equation}
\frac{d(|O'B'|)}{d\sigma}\ge -\frac{|O^*A^*|}{2}-\frac{|A^*D^*|\sin(|A^*O^*|)}{\sin(|A^*D^*|)}\sin(\frac{\pi}{3}) \cot(\gamma),
\label{d_OB}
\end{equation}
where $\gamma=\angle B'D'O'$.

\end{lemma}
\begin{proof}
By the spherical law of cosines, we have
\begin{eqnarray*}
&&\cos(|O'B'|)=\cos(|A'B'|)\cos(|A'O'|)+\sin(|A'B'|)\sin(|A'O'|)\cos(\angle B'A'O') \\
&=&\cos(|A'B'|)\cos(|A'O'| )+\sin(|A'B'|)\sin(|A'O'|)\frac{\cos(|B'D'|)-\cos(|A'B'|)\cos(|A'D'|)}{\sin(|A'B'|)\sin(|A'D'|)} \\
&=&\cos(|A'B'|)\frac{\sin(|O'D'|)}{\sin(|A'D'|)} +\cos(|B'D'|)\frac{\sin(|A'O'|)}{\sin(|A'D'|)}.
\end{eqnarray*}
Thus,
\begin{equation}
\cos(|O'B'|)=\cos(|A^*B^*|)\frac{\sin(|O^*D^*|  \sigma)}{\sin(|A^*D^*|  \sigma)} +\cos(|B^*D^*|)\frac{\sin(|A^*O^*|  \sigma)}{\sin(|A^*D^*|  \sigma)}.
\label{cos_OB}
\end{equation}
By the law of cosines, it follows that
\begin{equation}
\cos(|A^*B^*|) =\cos(|B'O'|)\cos(|A^*O^*|\sigma)-\sin(|B'O'|)\sin(|A^*O^*|\sigma)\cos(\alpha)
\label{cos_AB}
\end{equation}
and
\begin{equation}
\cos(|B^*D^*|)=\cos(|B'O'|)\cos(|O^*D^*|\sigma)+\sin(|B'O'|)\sin(|O^*D^*|\sigma)\cos(\alpha).
\label{cos_BD}
\end{equation}
Now, by taking derivatives on (\ref{cos_OB}) with respect to $\sigma$, and a simplification using (\ref{cos_AB}) and (\ref{cos_BD}), we have
\begin{eqnarray*}
&& -\sin(|O'B'|)\frac{d(|O'B'|)}{d\sigma} \\
&=&\frac{|O^*D^*| \cos(|O^*D^*| \sigma)\sin(|A^*D^*|\sigma)-|A^*D^*| \sin(|O^*D^*|\sigma)\cos(|A^*D^*|\sigma)}{\sin^2 (|A^*D^*|\sigma)}\cos(|A^*B^*|)\\
& & +\frac{|O^*A^*| \cos(|O^*A^*|\sigma)\sin(|A^*D^*|\sigma)-|A^*D^*| \sin(|O^*A^*|\sigma)\cos(|A^*D^*|\sigma)}{\sin^2 (|A^*D^*|\sigma)}\cos(|B^*D^*|)\\
&=&\frac{\cos(|B'O'|)}{\sin (|A^*D^*|\sigma)}[|A^*D^*|\sin(|O^*D^*|\sigma)\sin(|A^*O^*|\sigma)] +\frac{\sin(|B'O'|)}{\sin(|A^*D^*|\sigma)}\cos(\alpha)\\ 
&&\times[|O^*A^*| \cos(|O^*A^*|\sigma)\sin(|O^*D^*|\sigma)-|O^*D^*| \cos(|O^*D^*|\sigma)\sin(|A^*O^*|\sigma)]. 
\end{eqnarray*}

Thus, by means of (\ref{cotangent}), it follows that
\begin{eqnarray*}
&&-\frac{d(|O'B'|)}{d\sigma} \\
&=&\frac{|A^*D^*|\sin(|O^*D^*|\sigma)\sin(|A^*O^*|\sigma)}{\sin (|A^*D^*|\sigma)}\cot(|B'O'|)\\
& & +\frac{|O^*A^*| \cos(|O^*A^*|\sigma)\sin(|O^*D^*|\sigma)-|O^*D^*| \cos(|O^*D^*|\sigma)\sin(|A^*O^*|\sigma)}{\sin(|A^*D^*|\sigma)}\cos(\alpha)\\
& =&\frac{|A^*D^*|\sin(|O^*D^*|\sigma)\sin(|A^*O^*|\sigma)}{\sin (|A^*D^*|\sigma)}(\cot(|O^*D^*|\sigma) \cos(\alpha)+\cot(\gamma)\frac{\sin(\alpha)}{\sin(|O^*D^*|\sigma)}) \\
&&+\frac{|O^*A^*| \cos(|O^*A^*|\sigma)\sin(|O^*D^*|\sigma)-|O^*D^*| \cos(|O^*D^*|\sigma)\sin(|A^*O^*|\sigma)}{\sin(|A^*D^*|\sigma)}\cos(\alpha)\\
& =&|O^*A^*| \cos(\alpha)+\frac{|A^*D^*|\sin(|A^*O^*|\sigma)}{\sin (|A^*D^*|\sigma)}\sin(\alpha)\cot(\gamma).
\end{eqnarray*}
As a result, we have for all $\sigma \in[0,1]$,
\begin{equation}
\label{d_OB_1}
\frac{d(|O'B'|)}{d\sigma}=-|O^*A^*| \cos(\alpha)-\frac{|A^*D^*|\sin(|A^*O^*| \sigma)}{\sin(|A^*D^*| \sigma)}\sin(\alpha) \cot(\gamma).
\end{equation}
Now, we want to show that $\frac{d^2(|O'B'|)}{d\sigma^2} \le 0$ for all $\sigma \in[0,1]$. Let $p(\sigma)=\cos(|O'B'|)$. Then, 
\begin{eqnarray}
\frac{d^2(|O'B'|)}{d\sigma^2}&=&\frac{-p''(\sigma)}{\sqrt{1-p(\sigma)^2}}-\frac{(p'(\sigma))^2p(\sigma)}{(\sqrt{1-p(\sigma)^2})^3}.
\end{eqnarray}
Now, we investigate $p(\sigma)$ via the equation (\ref{cos_OB}). Clearly, when $$\max\{|A^*B^*|, |A^*C^*|, |B^*C^*|\}\le \frac{\pi}{2},$$ both $\cos(|A^*B^*|)\ge 0$ and $\cos(|B^*D^*|)\ge 0$. On the other hand, direct calculations give that when $b\ge a>0$,
\begin{eqnarray*}
\frac{d^2}{dx^2}\left(\frac{\sin(ax)}{\sin(bx)}\right)=\frac{\sin(ax)}{\sin(bx)}[(b\cot(bx)-a\cot(ax))^2+(\frac{b^2}{\sin^2(bx)}-\frac{a^2}{\sin^2(ax)})]\ge 0.
\end{eqnarray*}
Thus, by equation (\ref{cos_OB}), it follows that  $p(\sigma)\ge 0$ and $p''(\sigma)\ge 0$ for all $\sigma\le 1$. This shows that $\frac{d^2(|O'B'|)}{d\sigma^2} \le 0$.
As a result, when $0\le \sigma \le 1$,
\begin{equation*}
\frac{d(|O'B'|)}{d\sigma}\ge \frac{d(|O'B'|)}{d\sigma}\vert_{\{\sigma=1\}}=-\frac{|O^*A^*|}{2}-\frac{|A^*D^*|\sin(|A^*O^*|)}{\sin(|A^*D^*|)}\sin(\frac{\pi}{3}) \cot(\gamma).
\end{equation*}
Here, in the last equality, we used (\ref{d_OB_1}) and the fact that $O^*$ is the Fermat's point.
\end{proof}

\begin{lemma}
\label{lemma_positive_k}
Let $f$ be the function as given in (\ref{function_f}). When $k>0$, then $f(\sigma) \le f(1)$ whenever $\sigma  \in [0,1]$.
\end{lemma}

\begin{proof}
Without losing generality, we may assume that $k=1$. So for any $\sigma \le 1$, by (\ref{d_OB}), 
\begin{eqnarray*}
f'(\sigma)&=&\frac{d (|A'O'|+|B'O'|+|C'O'|)}{d\sigma} \\
&\ge & |A^*O^*|-\frac{|A^*O^*|}{2}-\frac{|A^*D^*|\sin(|A^*O^*|)}{\sin(|A^*D^*|)}\sin(\frac{\pi}{3}) \cot(\gamma)\\
& & \quad - \quad\frac{|A^*O^*|}{2}-\frac{|A^*D^*|\sin(|A^*O^*|)}{\sin(|A^*D^*|)}\sin(\frac{\pi}{3}) \cot(\pi-\gamma)=0.\\
\end{eqnarray*}
Thus, $f(\sigma) \le f(1)$.
\end{proof}

Now, we consider the case of $k<0$.

\begin{lemma} 
Using the same notations as in the Lemma \ref{spherical_lemma} except that the unit sphere $\mathbb{S}^2$ is replaced by the hyperbolic plane $(\mathbb{H}^2, |\cdot|_{h})=(M^2_{-1}, |\cdot|_{-1})$. 
If
\begin{equation}
\max\{|A^*B^*|_{h}, |A^*C^*|_{h}, |B^*C^*|_{h}\}\le 1.3877,
\label{condition_hyp}
\end{equation}
then for all $\sigma \in[0,1]$, 
\begin{equation}
\frac{d(|O'B'|_{h})}{d\sigma}\ge -\frac{|O^*A^*|_{h}}{2}-\frac{|A^*D^*|_{h}\sinh(|A^*O^*|_{h})}{\sinh(|A^*D^*|_{h})}\sin(\frac{\pi}{3}) \cot(\gamma),
\label{d_OB2}
\end{equation}
where $\gamma=\angle B'D'O'$.
\end{lemma}
\begin{proof}
In the hyperbolic case, analogous statements as in the proof of Lemma \ref{spherical_lemma} still give
\begin{equation}
\cosh(|O'B'|_{h})=\cosh(|A^*B^*|_{h})\frac{\sinh(|O^*D^*|_{h}  \sigma)}{\sinh(|A^*D^*|_{h}  \sigma)} +\cosh(|B^*D^*|_{h})\frac{\sinh(|A^*O^*|_{h}  \sigma)}{\sinh(|A^*D^*|_{h}  \sigma)},
\label{cosh_OB}
\end{equation}
and for all $\sigma \in[0,1]$,
\begin{equation}
\frac{d(|O'B'|_{h})}{d\sigma}=-|O^*A^*|_{h} \cos(\alpha)-\frac{|A^*D^*|_{h}\sinh(|A^*O^*|_{h} \sigma)}{\sinh(|A^*D^*|_{h} \sigma)}\sin(\alpha) \cot(\gamma).
\end{equation}
To prove (\ref{d_OB2}), it is sufficient to show that for all $\sigma \in[0,1]$, $\frac{d^2(|O'B'|_{h})}{d\sigma^2} \le 0$. The proof of this fact requires some further works.

Let $p(\sigma)=\cosh(|O'B'|_{h})$. Then, 
\begin{eqnarray}
\frac{d^2(|O'B'|_{h})}{d\sigma^2}&=&\frac{p''(\sigma)}{\sqrt{p(\sigma)^2-1}}-\frac{(p'(\sigma))^2p(\sigma)}{(\sqrt{p(\sigma)^2-1})^3}.
\end{eqnarray}
Together with equation (\ref{cosh_OB}),  Lemma \ref{hyperbolic_cosh} below implies that $p''(\sigma)\le 0$ under the condition (\ref{condition_hyp}). Since $p(\sigma)> 1$, it follows that $\frac{d^2(|O'B'|_{h})}{d\sigma^2} \le 0$.
\end{proof}

\begin{lemma}
\label{hyperbolic_cosh}
Suppose $0<a\le b \le 1.3877$, then for all $x\in [0,1]$, it holds that
\begin{eqnarray*}
\frac{d^2}{dx^2}\left(\frac{\sinh(ax)}{\sinh(bx)}\right)\le 0.
\end{eqnarray*}

\end{lemma}
\begin{proof}
Direct calculations give that
\begin{eqnarray*}
\frac{d^2}{dx^2}\left(\frac{\sinh(ax)}{\sinh(bx)}\right)=\frac{\sinh(ax)}{\sinh(bx)} g(x),
\end{eqnarray*}
where 
\begin{eqnarray*}
g(x)&=&(b\coth(bx)-a\coth(ax))^2+\frac{b^2}{\sinh^2(bx)}-\frac{a^2}{\sinh^2(ax)}\\
&=&2b \coth(bx)(b\coth(bx)-a\coth(ax))+a^2-b^2. \\
\end{eqnarray*}
Note that on the interval $(-\pi, \pi)$, the function $\coth(x)$ can be expressed as its Taylor series
\begin{eqnarray*}
\coth(x)=\frac{1}{x}+\sum_{n=1}^{\infty}\frac{2^{2n}B_{2n}}{(2n)!}x^{2n-1}=\frac{1}{x}+\frac{x}{3}-\frac{x^3}{45}+\frac{2x^5}{945}-\cdots,
\end{eqnarray*}
where $B_{2n}$ is the $2n$-th Bernoulli number given by
\[B_{2n}=(-1)^{n+1}\frac{2(2n)!}{(2\pi)^{2n}} (1+\frac{1}{2^{2n}}+\frac{1}{3^{2n}}+\frac{1}{4^{2n}}+\cdots).\]
Thus, when $0 \le ax\le bx<\pi$,
\begin{eqnarray*}
&& 0\le b\coth(bx)-a\coth(ax)\\
&&=b(\frac{1}{bx}+\frac{bx}{3}-\frac{b^3x^3}{45}+\frac{2b^5x^5}{945}-\cdots)-a(\frac{1}{ax}+\frac{ax}{3}-\frac{a^3x^3}{45}+\frac{2a^5x^5}{945}-\cdots) \\
&&=\frac{x}{3}(b^2-a^2)-\frac{x^3}{45}(b^4-a^4)+\frac{2x^5}{945}(b^6-a^6)-\cdots \\
&&\le \frac{x}{3}(b^2-a^2)-\frac{x^3}{45}(b^4-a^4)+\frac{2x^5}{945}(b^6-a^6).
\end{eqnarray*}
Also,
\begin{eqnarray*}
 0\le b \coth(bx)=b(\frac{1}{bx}+\frac{bx}{3}-\frac{b^3x^3}{45}+\frac{2b^5x^5}{945}-\cdots) \le b(\frac{1}{bx}+\frac{bx}{3})=\frac{1}{x}+\frac{b^2x}{3}.
\end{eqnarray*}
As a result,
\begin{eqnarray*}
g(x)&=&2b \coth(bx)(b\coth(bx)-a\coth(ax))+a^2-b^2\\
&& \le 2(\frac{1}{x}+\frac{b^2x}{3})[\frac{x}{3}(b^2-a^2)-\frac{x^3}{45}(b^4-a^4)+\frac{2x^5}{945}(b^6-a^6)]+a^2-b^2\\
&&=\frac{a^2-b^2}{3}[1+\frac{2a^2-8b^2}{15}x^2+\frac{12b^4+12a^2b^2-2a^4}{315}x^4-\frac{4b^2(a^4+b^4+a^2b^2)}{945}x^6]\\
&&\le\frac{a^2-b^2}{3}[1+\frac{0-8b^2}{15}x^2+\frac{12b^4+0-2b^4}{315}x^4-\frac{4b^2(b^4+b^4+b^4)}{945}x^6]\\
&&=\frac{a^2-b^2}{945}[315-168(bx)^2+10(bx)^4-4(bx)^6] \le 0,
\end{eqnarray*}
whenever $0<a\le b$ and $ (bx)^2\le 1.9257$. Thus,  when $0<b\le \sqrt{1.9257}=1.3877$, it holds that $g(x)\le 0$ for all $x\in [0,1]$ and for all $0<a\le b$.
\end{proof}
\begin{lemma}
\label{lemma_negative_k}
Let $f$ be the function as given in (\ref{function_f}). When $k<0$, then $f(\sigma) \le f(1)$ whenever $\sigma  \in [0,1]$.
\end{lemma}

\begin{proof}
Follows from an analogous proof of Lemma \ref{lemma_positive_k}, via equation (\ref{d_OB2}).
\end{proof}

%
%
%
%
%
%
%
%
%

\bibliographystyle{amsplain}

\begin{thebibliography}{9}
 \bibitem{BH}
 M. R. Bridson and A. Haefliger, Metric Spaces of Non-positive Curvature, Springer-Verlag, Berlin Heidelberg, 1999.
 \bibitem{metricbook}
  D. Burago, Y. Burago, and S. Ivanov, A Course in Metric Geometry, in: Graduate Studies in Math., vol. 33, Amer. Math. Soc., Providence, RI, 2001.
   \bibitem{forman} Robin Forman, Bochner’s method for cell complexes and combinatorial Ricci curvature,
Discrete Comput. Geom. 29:3 (2003), 323–-374.
\bibitem{Lee}
John M. Lee, Riemannian Manifolds: An Introduction to Curvature. Graduate Texts in Mathematics, Vol. 176. (1997)

\bibitem{saucan} Emil Saucan, Curvature – Smooth, Piecewise-Linear and Metric, book chapter in What is Geometry?, Advanced Studies in Mathematics and Logic, 237-268, Polimetrica, Milano, 2006.



\bibitem{sullivan}
John M. Sullivan,  Curvatures of Smooth and Discrete Surfaces. Discrete Differential Geometry, Oberwolfach Seminars 38, Birkhauser, 2008.
\end{thebibliography}



\end{document}